\documentclass[11pt,leqno]{amsart}
\usepackage{amscd}
\usepackage{amssymb}
\usepackage{amsfonts}
\usepackage{latexsym}
\usepackage{verbatim}
\usepackage{amsthm}
\usepackage{mathrsfs}

\renewcommand{\,}{\kern1pt}
\newcommand{\sq}{\hbox{\small$\sqrt3$\,}}
\newcommand{\aqH}{AQH}
\newcommand{\bbb}{\beta}
\newcommand{\aaa}{\alpha}
\newcommand{\sss}{\sigma}
\newcommand{\pp}{\mathfrak{p}}
\renewcommand{\ss}{\mathfrak{s}}
\newcommand{\tr}{\mathop{\mathrm{tr}}}
\newcommand{\suml}{\textstyle\sum\limits}

\numberwithin{equation}{section}

\theoremstyle{plain}
\newtheorem{theorem}{Theorem}[section]

\newtheorem{lemma}[theorem]{Lemma}
\newtheorem{proposition}[theorem]{Proposition}
\newtheorem{corollary}[theorem]{Corollary}

\newcommand{\wt}[1]{\widetilde{#1}}

\newcommand\SU{{\rm SU}}
\newcommand\Sp{{\rm Sp}}
\newcommand\SO{{\rm SO}}
\newcommand{\w}{\omega}

\begin{document}
\title[Nearly quaternionic structure on $\SU(3)$]
{A nearly quaternionic structure on $\SU(3)$}
\author[\'O.\ Maci\'a]{\'Oscar Maci\'a}

\address{Dipartimento di Matematica
\\ Politecnico di Torino \\ Corso Duca degli Abruzzi 24, 10129 Torino,
Italy} \email{oscarmacia@calvino.polito.it}
\begin{abstract}
It is shown that the compact Lie group $\SU(3)$ admits an
$\Sp(2)\Sp(1)$ structure whose distinguished 2-forms
$\w_1,\w_2,\w_3$ span a differential ideal. This is achieved by
first reducing the structure further to a subgroup isomorphic to
$\SO(3)$.
\end{abstract}
\maketitle

\newcommand\C{{\mathbb C}}
\newcommand\R{{\mathbb R}}

\newcommand\GL{{\rm GL}}
\newcommand\SL{{\rm SL}}
\newcommand\su{\mathfrak{su}}
\newcommand\so{\mathfrak{so}}
\newcommand\fsp{\mathfrak{sp}}
\newcommand\cI{\mathscr{I}}
\newcommand\cW{\mathcal{W}}

%
%

\parskip2pt

\section{introduction}

An almost quaternionic Hermitian (\aqH) manifold is a Riemannian
$4n$-manifold $\{M,g\}$ admitting a $\Sp(n)\Sp(1)$-structure, that is a
reduction of its frame bundle to a subbundle whose structure group is
the subgroup $\Sp(n)\Sp(1)$ of $\SO(4n).$ This means that $\{M,g\}$ is
equipped locally with a triple of almost complex structures
$\{I_1,I_2,I_3\}$ that behave like the imaginary quaternions $i,j,k$,
and are compatible with the metric. The almost complex structures
$I_i$ generate a subbundle $\cI$ of endomorphisms of $TM$.

Following the method initiated by Gray \& Hervella for the study of
almost Hermitian manifolds \cite{Gray78}, the space
\[
\R^{4n}\otimes(\fsp(n)\oplus\fsp(1))^\perp
\]
of intrinsic torsion tensors decomposes into irreducible modules under
the action of $\Sp(n)\Sp(1)$, giving rise to a natural classification
of \aqH\ manifolds. The intrinsic torsion can be identified with the
Levi-Civita derivative $\nabla\Omega$, where
\[\Omega=\suml_{j=1}^3\omega_j\wedge\omega_j,\]
is the fundamental 4-form, defined locally in terms of the 2-forms
given by $\omega_i(X,Y)=g(I_iX,Y)$. Conditions describing the
intrinsic torsion classes can be studied accordingly. In the general
case, for $n>2,$ there exist six irreducible components of intrinsic
torsion. But for $n=2$, only four components arise, giving $2^4=16$
classes of \aqH\ $8$-manifolds, with a closer analogy to the almost
Hermitian complex case.

If the intrinsic torsion vanishes the \aqH\ manifold is said to be
quaternionic K\"ahler, the holonomy reduces to $\Sp(n)\Sp(1)$ and the
manifold is Einstein. In 1989, Swann \cite{Swann89} proved
\begin{theorem}\label{qK}
Let $\{M,g,\cI\}$ be an \aqH\ $4n$-manifold, $n>2,$ with
fundamental 4-form $\Omega.$ Then, it is quaternionic K\"ahler if
and only if $\Omega$ is closed.\\ For $n=2$, $\{M,g,\cI\}$ is
quaternionic K\"ahler if and only if
\begin{enumerate}
\item The  fundamental 4-form is closed, $d\Omega=0.$
\item The set $\{\w_1,\w_2,\w_3\}$ of $2$-forms generates a
  differential ideal.
\end{enumerate}
\end{theorem}

\noindent Condition (2) means that there exists a $3\times3$ matrix
$(\bbb_i^j)$ of 1-forms such that
\begin{equation}\label{ideal}
d\omega_i=\suml_{j=1}^3\bbb_i^j\wedge\omega_j,
\end{equation}
The condition itself is easily seen to be dependent only on $\cI$ and
not the choice of basis (see Section~3).\smallbreak

In the case $n=2$, this result left open the existence question for
manifolds satisfying one of the conditions but not the other. One
affirmative answer was provided in 2001 by Salamon \cite{Salamon00}
with an example of an \aqH\ 8-manifold with closed, but non-parallel,
fundamental 4-form. This `almost parallel' manifold is a product of a
3-torus with a 5-dimensional nilmanifold, and similar examples were
found by Giovannini \cite{Giovannini}.

In the present paper we deal with the complementary case, namely
$8$-manifolds with $\Sp(2)\Sp(1)$-structure for which the $\cI$
generates a differential ideal, but for which the 4-form $\Omega$ is
not closed.

The paper is organized as follows. In Section~2, we discuss the four
components of intrinsic torsion and their relationship with the
differential ideal condition. Other properties of this condition are
discussed in Section~3, which enables us to re-formulate the
quaternionic K\"ahler condition. A one-parameter family of
quaternionic structures is defined in Section~4 on $\SU(3)$, endowed
with a compatible deformation $g_\lambda$ of its bi-invariant metric.
In Section~5, it is shown that for specific choices of the parameter
$\lambda,$ $\{\SU(3),g_\lambda,\cI\}$ is \aqH\ and satisfies condition
(2), but not (1) in Theorem~\ref{qK}.

Some of the computations in Sections 4 and 5 were effectively carried
out using \emph{Mathematica} and the differential forms package
\emph{scalarEDC} \cite{Bonanos}.

%
%

\section{Intrinsic torsion and reduction to $\SO(3)$}

We will use the $E$-$H$ formalism described in \cite{Salamon82}.
Suppose that $n\geq 2$. Let $E$ (respectively, $H$) denote the
basic complex representation of $\Sp(n)$ (respectively, $\Sp(1)$),
with highest weight $(1,0,\dots,0)$ (resp.  $(1)$), such that
$E\simeq \C^{2n}$ (resp. $H\simeq \C^2)$. We denote the
$\Sp(n)$-module with highest weight $(1,\dots,1,0,\dots,0),$ (with
$r$ $1$'s and $n-r$ $0$'s) by $\Lambda^r_0E.$ Also, $S^r E$
(respectively, $S^r H$) will denote the $\Sp(n)$-module
(respectively, $\Sp(1)$) with highest weight $(r,0,\dots,0)$
(respectively, $(r)$). Finally, let $K$ be the $\Sp(n)$-module
with highest weight $(2,1,0,\dots,0).$

The fundamental 4-form of the $\Sp(n)\Sp(1)$-structure is the
distinguished element arising the decomposition of $\Lambda^4T^*$
under the action of $\Sp(n)\Sp(1),$ where
\begin{equation}\label{TEH}
T^*\otimes_\R\C=E\otimes H
\end{equation}
represents the complexified cotangent space. The intrinsic torsion
$\xi=\nabla\Omega$ of an $\Sp(n)\Sp(1)$-structure is described by the
following

\begin{theorem} \emph{(}Swann, \cite{Swann89}\emph{)} The intrinsic torsion of an \aqH\
$4n$-manifold, $n\ge2$ can be identified with an element in the
space
$$\left(\Lambda^3_0E\oplus K\oplus E\right)\otimes\left(H \oplus
  S^3H\right).$$ For $n=2,$ the intrinsic torsion belongs
  to \begin{equation}\label{modules}
E\,S^3H\oplus K\,S^3H\oplus KH\oplus EH.\end{equation}
\end{theorem}\medskip

Various examples of \aqH\ 8-manifolds with different types of
intrinsic torsion are known (see for example Cabrera \& Swann
\cite{Cabrera2007}):
\begin{corollary} An \aqH\ 8-manifold $M$ is quaternionic if and only if
$\xi\in KH\oplus EH$.
\end{corollary}

\noindent The adjective `quaternionic' here means that the
underlying $\GL(2,\mathbb{H})\Sp(1)$ admits a torsion-free
connection. As stated, this condition is characterized by the
absence of the $\Sp(1)$ module $S^3H$, and ensures that the
`twistor space' associated to $M$, defined in
\cite[Ch.~14]{Besse}, is a complex manifold.\smallbreak

\begin{corollary} The fundamental 4-form of an \aqH\ 8-manifold is
closed, i.e., $M$ is almost parallel if and only if $\xi\in KS^3H.$
\end{corollary}
\begin{corollary}\label{idealclass} The 2-forms $\{\omega_i\}$ of an
\aqH\ 8-manifold generate a differential ideal if and only if $\xi\in
ES^3H\oplus EH,$
\end{corollary}

\noindent These two corollaries represent the conditions stated in
Theorem \ref{qK} for $n=2$, namely (1) and (2) respectively. The
author does not know of any example in the literature of a non
quaternionic K\"ahler 8-manifold in the class described by Corollary
\ref{idealclass}.\smallbreak

One can establish a useful analogy between the 16 Gray--Hervella
classes of almost Hermitian $2n$-manifolds and the 16 classes of
\aqH 8-manifolds. This is best done by indicating the
$\Sp(2)\Sp(1)$-modules in Equation~\eqref{modules} by the symbols
$\cW_1,\dots,\cW_4$ respectively. This ensures that, in both
cases, `integrability' corresponds to vanishing of the
$\cW_1\oplus\cW_2$ component, and a conformal change in the metric
modifies in an essential way only the $\cW_4$ component.

On the other hand, the class of \aqH\ 8-manifolds with intrinsic
torsion belonging to $\cW_1\oplus\cW_4$ in Corollary
\ref{idealclass} is the \aqH\ analogue of the family of almost
Hermitian manifolds containing nearly K\"ahler and locally
conformal K\"ahler manifolds described by Butruille
\cite{Butruille} and Cleyton \& Ivanov \cite{Cleyton}. For more
details on the classification of \aqH\ manifolds see Cabrera \&
Swann \cite{Cabrera2007}.

Consider the homomorphism
\[\phi\colon \Sp(1)\to \Sp(2)\times \Sp(1)\]
defined by $\phi(g)=(i(g),g)$, where $i\colon
\Sp(1)\hookrightarrow \Sp(2)$ is the inclusion whereby $\Sp(1)$
acts irreducibly on $\C^4$. By definition, $\Sp(2)\Sp(1)$ is a
$\mathbb{Z}_2$ quotient of $\Sp(2)\times \Sp(1)$ whose kernel is
generated by $(-\mathbf1,-\mathbf1)$. Therefore, $\phi$ induces an
inclusion
\begin{equation}\label{SO3}
\SO(3)=\Sp(1)/\mathbb{Z}_2\to \Sp(2)\Sp(1)\subset \SO(8),
\end{equation}
and in this paper we shall effectively be considering such
$\SO(3)$ structures on 8-manifolds.

Let $M$ be an 8-manifold with an $\SO(3)$-structure compatible
\eqref{SO3}.  Using the well-known formula
$$S^pH\otimes S^qH\simeq\bigoplus_{n=0}^{\min(p,q)}S^{p+q-2n}H,$$ the
complexified tangent space \eqref{TEH} now splits as
\begin{equation}\label{S3H}
S^3H\otimes H \simeq S^2H\oplus S^4H.
\end{equation}
The underlying quaternionic action is defined by a suitable inclusion
of $S^2H$ in the space of anti-symmetric endomorphisms of the tangent
space, isomorphic to
\[ \Lambda^2T^* \simeq 2S^6H\oplus S^4H\oplus 3S^2H.\]
Its image is a coefficient bundle of purely imaginary quaternions.

Relative to \eqref{SO3}, we have $E\simeq S^3H$ where $H$ now
denotes the spin representation of $\mathrm{Spin}(3)$. It follows
that $\Lambda_0^2E\simeq S^4H$ and $K\simeq S^7H\oplus S^5H\oplus
S^1H$.  The intrinsic torsion space \eqref{modules} then
decomposes as follows
\[\begin{array}{ccccl}
\cW_1 &\!=\!& E\,S^3H &\!\simeq\!&
S^6H\oplus S^4H\oplus S^2H\oplus S^0H\\[2pt]
\cW_2 &\!=\!& K\,S^3H &\!\simeq\!&
S^{10}H\oplus 2S^8H\oplus 2S^6H\oplus 3S^4H\oplus 2S^2H\\[2pt]
\cW_3 &\!=\!& KH &\!\simeq\!&
S^8H\oplus 2S^6H\oplus S^4H\oplus S^2H\oplus S^0H\\[2pt]
\cW_4 &\!=\!& EH &\!\simeq\!& S^4H\oplus S^2H.
\end{array}\]
These isomorphisms reveal the presence of a 2-dimensional space of
$\SO(3)$-invariant tensors. Our aim is to describe an example with
intrinsic torsion in the summand $S^0H$ in $\cW_1$. Such a `nearly
quaternionic structure' will be found on $\SU(3)$, although the
general theory of $\SO(3)$ structures on 8-manifolds will be
pursued elsewhere \cite{Chiossi}.

Manifolds with $\SO(3)$ structure as in \eqref{S3H} are special
cases of those considered by Swann \cite{Swann91}, and later
Gambioli \cite{Gambioli08a}, in the context of nilpotent coadjoint
orbits of a complex Lie group. The same structure arises naturally
on the total space of a rank 3 vector bundle over $\SU(3)/\SO(3)$
\cite{Conti05,Gambioli08b}, and can be analysed with the methods
of Conti \cite{Conti07}.

For the special case of $\SU(3)$,  the tangent space can be
identified with the Lie algebra
\begin{equation}\label{su3}
 \su(3)=\ss\oplus\pp
\end{equation}
of complex anti-Hermitian $3\times3$ matrices. Here $\ss$ is an
abbreviation for the subalgebra $\so(3)$ of real anti-symmetric
matrices, whereas $\pp$ is the space of matrices of the form $iS$
with $S$ symmetric and trace-free. Observe that the decomposition
\eqref{su3} is consistent with \eqref{S3H} with $\ss\simeq S^2H$
and $\pp\simeq S^4H$. On an 8-manifold, in view of \eqref{su3},
any reduction to $\SO(3)$ determines not only an almost
quaternionic structure, but also a $\mathrm{PSU}(3)$-structure in
the sense of Hitchin \cite{Hitchin}.

In Section~4, we shall define endomorphisms $I_i$ that act on
\eqref{su3} extending the adjoint representation on $\ss$.

%
%

\section{The ideal condition}

We suppose throughout this section that $\{\w_1,\w_2,\w_3\}$ is a
locally-defined set of 2-forms associated to a basis $\{I_1,I_2,I_3\}$
of $\cI$ on an \aqH\ 8-manifold, and that the differential ideal
condition \eqref{ideal} is satisfied.

Any two bases $\{\omega_i\}$, $\{\wt\omega_i\}$ are related by
a gauge transformation of the form
\[\wt\w_i=\suml_{j=1}^3 A_i^j\w_j,\]
with $A=(A_i^j)$ taking values in $\SO(3)$ at each point. Then we
can write
\[ d\wt\w_i=\suml_{j=1}^3 \wt\bbb^j_i\wedge \wt\w_j,\]
where
\[\wt\bbb_i^j=(A^{-1})^j_l\,dA_i^l+(A^{-1})_l^j\bbb_k^lA_i^k,\]
with summation over repeated indices. The matrix $\bbb$ therefore
transforms as
\begin{equation}\label{connectionalpha}
\wt\bbb = A^{-1}dA + \mathrm{Ad}(A^{-1})\bbb.
\end{equation}
It follows that $\bbb$ represents a connection on the rank 3
vector bundle, isomorphic to $\cI$, generated by the
$\{\omega_i\}$. However, this connection does not reduce to
$\SO(3)$ unless $\bbb$ is anti-symmetric, a point we discuss next
before a brief analysis of curvature.

Consider the decomposition
\begin{equation}\label{bas}
\bbb=\aaa+\sss,
\end{equation}
where $\aaa_i^j=\frac12(\bbb_i^j-\bbb_j^i)$ and
$\sss_i^j=\frac12(\bbb_i^j+\bbb_j^i)$ are the anti-symmetric and
symmetric parts. The fact that $A\in \SO(3),$ implies that $A^{-1}dA$
lies in the Lie algebra $\mathfrak{so}(3)$ of anti-symmetric matrices.
Given that $\mathrm{Ad}$ preserves the decomposition \eqref{bas}, we
see that the symmetric part $\sss$ transforms as a tensor:
\[\wt\sss = \mathrm{Ad}(A^{-1})\sss = A^{-1}\sss A,\]
in contrast to $\bbb$.

The tensor represented by $\sss$ can be identified with the remaining
non-zero components
\[ \cW_1\oplus\cW_4\simeq ES^3H\oplus EH\]
of the intrinsic torsion, or equivalently $d\Omega$. Indeed,
\begin{equation}\label{dOmega}
d\Omega=2\suml_{i=1}^3d\omega_i\wedge\omega_i
= 2\suml_{i,j=1}^3\bbb_i^j\wedge\omega_i\wedge\omega_j
= 2\!\suml_{i,j=1}^3\sss_i^j\wedge\omega_i\wedge\omega_j.
\end{equation}
We can easily identify the component in $\cW_1$:

\begin{lemma}\label{tr} If an $\Sp(2)\Sp(1)$-structure satisfies
\eqref{ideal} then its intrinsic torsion belongs to $\cW_1$ if and
only if $\tr(\bbb)=\bbb_1^1+\bbb_2^2+\bbb_3^3$ vanishes.
\end{lemma}

\begin{proof} The $\cW_4$ part of the torsion is represented by
the component of $d\Omega$ in the submodule $EH$ of
$\Lambda^5T^*$. But \eqref{dOmega} belongs to
\[ EH\otimes S^2(S^2H)\simeq E\,S^5H\oplus E\,S^3H\oplus EH,\]
and its $EH$ component can only be obtained by taking the trace over
each term $\omega_i\wedge\omega_j$, leaving us with
$2\tr(\bbb)=2\tr(\sss)$.\end{proof}

We can also use \eqref{dOmega} to give an alternative characterization
of quaternionic K\"ahler 8-manifolds. Theorem \ref{qK} implies the

\begin{corollary} Let $\{M,g,\cI\}$ be an \aqH\ $8$-manifold. It  is
quaternionic K\"ahler if and only if $\cI$ generates a differential
ideal with $\sss=0$, so that \eqref{ideal} applies with
$\bbb_i^j=-\bbb_j^i.$
\end{corollary}

Returning to \eqref{ideal}, we may consider the matrix
$B=(B_i^j)$ of curvature 2-forms associated to
the connection we have considered. These 2-forms arise in the
computation
\begin{equation}\label{Bij}
0 = d^2\omega_i =
\suml_j(d\bbb^j_i-\suml_k\bbb^k_i\wedge\bbb_k^j)\wedge\omega_k =
\suml_j B^j_i\wedge\omega_j,
\end{equation} which also provides a constraint on
them. In particular, they have no $S^2E$ component, because
$\Lambda^4T^*$ contains the module $S^2E\,S^2H$ \cite{Swann89}; thus
\[  B_i^j\in S^2H\oplus \Lambda_0^2E\,S^2H\subset \Lambda^2T^*.\]
But in contrast to the quaternionic K\"ahler case, there will in
general be a component of $B_i^j$ in $\Lambda_0^2E\,S^2H$. This
will be treated in a forthcoming paper.

%
%

\section{Quaternionic endomorphisms of $\mathfrak{su}(3)$}

Let $E_{ij}$ denote $3\times 3$ matrix with a $1$ in the $ij$
position, and $0$'s elsewhere. We adopt the following basis of the Lie
algebra \eqref{su3} of anti-Hermitian matrices:
\begin{equation}\label{choice}
\left\{\begin{array}{rcl}
e_1&=&i\left(E_{11}-E_{33}\right)\\[4pt]
e_2&=&\textstyle\frac1\sq i
      \left(-E_{11}+2E_{22}-E_{33} \right)\\[4pt]
e_3&=&i\left(E_{21}+E_{12}\right)\\[4pt]
e_4&=&i\left(E_{31}+E_{13}\right)\\[4pt]
e_5&=&i\left(E_{32}+E_{23}\right)\\[4pt]
e_6&=&\left(E_{21}-E_{12}\right)\\[4pt]
e_7&=&\left(E_{31}-E_{13}\right)\\[4pt]
e_8&=&\left(E_{32}-E_{23}\right)
\end{array}\right.
\end{equation}
The choice of basis has been taken such that it is conformal
relative to  minus the Killing form:
\[  \tr(e_ie_j)=-2\delta_{ij}.\]
Observe that $\{e_6,e_7,e_8\}$ is a basis of the subalgebra $\so(3)$
of real anti-symmetric matrices.

The Lie bracket of $\su(3)$ is defined by $[A,B]=AB-BA$. For the basis
above, they are given by
\[\begin{array}{lllll}
[e_1,e_2]=0 & & &\\{}
[e_1,e_3]=e_6 & [e_2,e_3]=-\sq e_6 & &\\{}
[e_1,e_4]=2e_7 & [e_2,e_4]=0 & [e_3,e_4]=e_8 &\\{}
[e_1,e_5]=e_8 & [e_2,e_5]=\sq e_8 &[e_3,e_5]=e_7 & [e_4,e_5]=e_6\\{}
[e_1,e_6]=-e_3 & [e_2,e_6]=\sq e_3 &[e_3,e_6]=e_1-\sq e_2& [e_4,e_6]=-e_5\\{}
[e_1,e_7]=-2e_4 & [e_2,e_7]=0 &[e_3,e_7]=-e_5 & [e_4,e_7]=2e_1\\{}
[e_1,e_8]=-e_5 & [e_2,e_8]=-\sq e_5 &[e_3,e_8]=-e_4 & [e_4,e_8]=e_3,
\end{array}\]
together with
\[\begin{array}{lllll}
[e_5,e_6]=e_4 & & & \\{}
[e_5,e_7]=e_3 & [e_6,e_7]=e_8 & &\\{}
[e_5,e_8]=e_1+\sq e_2 & [e_6,e_8]=-e_7 & [e_7,e_8]=e_6. &
\hskip80pt
\end{array}\]
We can regard these elements as left-invariant vector fields on
the Lie group $\SU(3)$.

Now let $\{e^1,\ldots,e^8\}$ be the dual basis
$\mathfrak{su}(3)^*$, or equivalently left-invariant 1-forms on
$\SU(3)$, so that $e^i(e_j)=\delta^i_j$. Using the Cartan formula,
we arrive at the exterior differential system\smallbreak

\begin{eqnarray*}
de^1 &=& -e^{36}-2e^{47}-e^{58},\\
de^2 & = & \sq(e^{36}-e^{58}),\\
de^3 &=& e^{16}-\sq e^{26}-e^{48}-e^{57},\\
de^4 &= & 2e^{17}+e^{38}-e^{56},\\
de^5&= & e^{18}+\sq e^{28}+e^{37}+e^{46},\\
de^6&= & -e^{13}+\sq e^{23}-e^{45}-e^{78},\\
de^7 &= &-2e^{14}-e^{35}+e^{68},\\
de^8&=& -e^{15}-\sq e^{25}-e^{34}-e^{67}.
\end{eqnarray*}
\smallbreak

Referring to \eqref{su3}, we shall use the notation $S^2H$ to indicate
the space of quaternionic endomorphisms, and identify it with $\ss$ in
a natural way by setting
\begin{equation}\label{ident}
I_1=e_8,\quad I_2=-e_7,\quad I_3=e_6.
\end{equation}
We are going to define an $\SO(3)$-equivariant linear mapping
\begin{equation}\label{equi}
S^2H\otimes(\ss\oplus\pp)\to(\ss\oplus\pp),
\end{equation}
by considering the associated four maps one at a time. In view
of the isomorphisms
\[\begin{array}{l}
S^2H\otimes S^2H\simeq S^4H\oplus S^2H\oplus S^0H,\\[3pt] S^2H\otimes
S^4H\simeq S^6H\oplus S^4H\oplus S^2H,
\end{array}\]
each of the four maps is uniquely determined up to a scalar multiple.

Any equivariant linear map \eqref{equi} must therefore be a
linear combination of the following four non-zero maps:
\[\begin{array}{lcll}
\phi_1 &:& S^2H\otimes\ss\to\ss\quad & (A,B)\mapsto [A,B]\\[3pt]
\phi_2 &:& S^2H\otimes\ss\to\pp & (A,B)\mapsto\textstyle
i\big(\{A,B\}-\frac23\tr(AB)\mathbf1\big)\\[3pt]
\phi_3 &:& S^2H\otimes \pp \rightarrow \ss & (A,C)\mapsto i \{A,C\}\\[3pt]
\phi_4 &:& S^2H\otimes \pp\to\pp & (A,C)\mapsto [A,C].
\end{array}\]
Here, $A\in S^2H$ is identified with an element of $\ss$ via
\eqref{ident}, $B\in\ss$, and $C\in\pp$. Also, $\{A,B\}=AB+AB$ is the
anti-commutator, and $\bf1$ denotes the $3\times3$ identity
matrix. Note that all the images on the right-hand side have zero
trace and are anti-Hermitian, as required.

\begin{proposition}\label{onepara} There is a one-parameter family of
 $\SO(3)$-invariant quaternionic actions on $\su(3).$
\end{proposition}

\begin{proof} Introducing a constant $\lambda_i$ for each $\phi_i$,
\eqref{equi} must be given by
\[\textstyle A\cdot X = \lambda_1[A,X^a] + i \lambda_2
\Big(\!\{A,X^a\}-\frac23\tr(AX^a)\mathbf1\!\Big) + i
\lambda_3\{A,X^s\} + \lambda_4[A,X^s],\] where $A\in S^2H$, and
$X^a=\frac12(X-X^t)$ and $X^s=\frac12(X+X^t)$ are the
(anti-)symmetric components of $X\in\su(3)$.

Recall the formula \eqref{ident} to identify endomorphisms with
elements of $\ss$. We first impose the identities
\begin{equation}\label{IIX}
I_i\cdot(I_i\cdot X)=-X,\qquad i=1,2,3.
\end{equation}
These can be used to find the $\lambda_k$ by making simple choices of
the matrix $X$. Calculations show that \eqref{IIX} holds fully
when
\begin{equation}\label{quat}\textstyle
\lambda_1=\frac12\epsilon,\qquad
\lambda_3=-\frac34(\lambda_2)^{-1},\qquad \lambda_4=-\frac12\epsilon,
\end{equation}
provided $\epsilon=\pm1$. We take $\epsilon=+1$, since this makes
the identity
\[ I_1\cdot(I_2\cdot X)=I_3\cdot X=-I_2\cdot(I_1\cdot X)\]
automatically valid (whereas $\epsilon=-1$ would give us
$I_2I_1=I_3$).
We can now parametrise the quaternionic structures
by $\lambda_2$, and the proof is complete.
\end{proof}

We shall denote by $\cI_\lambda$ the quaternionic structure defined by
\eqref{quat} in terms of the parameter $\lambda:=\lambda_2$. Recall
that $\{e^i\}$ is an orthonormal basis of $\su(3)$ for a multiple of
the Killing form. Next, we deform this metric by rescaling on the
subspace $\ss$.

\begin{proposition} The Riemannian metric
\begin{equation}\label{glambda}
g_\lambda = \suml_{i=1}^5 e^i\otimes e^i +
\frac43\lambda^2\suml_{i=6}^8 e^i\otimes e^i
\end{equation} is compatible with the structure $\cI_\lambda.$
\end{proposition}

\begin{proof}
If $\{I_i\}$ are the endomorphisms defined in
  \eqref{ident}, with the quaternionic action defined in Proposition
 \ref{onepara} we need to show that
\[ g_\lambda(I_i\cdot X,I_i\cdot Y)=g_\lambda(X,Y),\qquad i=1,2,3.\]
Since $\cI_\lambda$ is $\SO(3)$-invariant, and both subpsaces
$\ss$, $\pp$ are irreducible, this equation \emph{must} hold for
some choice of $\lambda$. The rest is a computation.\end{proof}

\section{The main result}

The formula for $A\cdot X$ in the proof of Proposition~\ref{onepara}
can be used to compute the endomorphisms $I_i$ explicitly. For
example, the action of $I_3$ on $\su(3)$ is given by
\begin{equation}\label{I18}
\left\{\begin{array}{rcl}
e_1 &\mapsto& \frac12e_3+\frac34\lambda^{-1}e_6\\[3pt]
e_2 &\mapsto& -\frac12\sq e_3+\frac14\sq\lambda^{-1}e_6\\[3pt]
e_3 &\mapsto& -\frac12e_1+\frac12\sq e_2\\[3pt]
e_4 &\mapsto& \frac12e_5-\frac34\lambda^{-1}e_8\\[3pt]
e_5 &\mapsto& -\frac12e_4+\frac34\lambda^{-1}e_7\\[3pt]
e_6 &\mapsto& -\lambda e_1-\frac1\sq\lambda e_2\\[3pt]
e_7 &\mapsto& -\frac12e_8-\lambda e_5\\[3pt]
e_8 &\mapsto& \frac12e_7+\lambda e_4.
\end{array}\right.
\end{equation}

We can use these formulas, and analogous ones for $I_1,I_2$ to prove

\begin{proposition}\label{ooo}
 A set of $2$-forms $\{\omega_i\}$ associated to the
  \aqH\ manifold $\{\SU(3),g_\lambda,\cI_\lambda\}$ is given by
\[\begin{array}{l}\textstyle
\omega_1 = \frac12(e^{15}+\sq e^{25}+e^{34}) + \lambda(\frac1\sq
e^{28}-e^{46}+e^{37}-e^{18}) -\frac23\lambda^2e^{67},\\[9pt]
\omega_2=-e^{14}-\frac12e^{35}+\lambda(\frac2\sq
e^{27}-e^{38}-e^{56})
-\frac23\lambda^2e^{68},\\[9pt]
\omega_3 = \frac12(e^{13}-\sq e^{23}+e^{45}) + \lambda(\frac1\sq
e^{26}-e^{48}+e^{57}+e^{16}) -\frac23\lambda^2e^{78}.
\end{array}\]
\end{proposition}

\begin{proof} The expression of $\omega_3$ follows easily from \eqref{I18}. For
example, using \eqref{glambda},
\[\textstyle \omega_3(e_1,e_6)=g_\lambda(I_3e_1,e_6)=
\frac34\lambda^{-1}g_\lambda(e_6,e_6)=\lambda,\] explaining the
coefficient of $e^{16}$. Minus the same coefficient is visible in the
expression for $I_3e_6$, which is consistent.

The expressions for $\omega_1,\omega_2$ follow in a similar way
from the computation of $I_1,I_2$ that we omit.
\end{proof}

We are now in a position to verify if and when the ideal condition
\eqref{ideal} holds. Since $\bbb=(\bbb_i^j)$ is a matrix of 1-forms,
we first impose the condition that its trace vanishes. By
Lemma~\ref{tr}, this amounts to assuming that the intrinsic torsion
lies in $\cW_1$.  Then $\bbb$ takes values in
\begin{equation}\label{sl3R}
\mathfrak{sl}(3,\R)=\ss\oplus i\pp,
\end{equation}
where $i\pp$ denotes the 5-dimensional space of real symmetric
trace-free matrices. The decomposition \eqref{sl3R} is merely the
Cartan dual of \eqref{su3} in the theory of symmetric spaces.

Once we express the 1-forms $\bbb_i^j$ in terms of the basis dual to
\eqref{choice}, we can regard $\xi\mapsto\bbb(\xi)$ as a linear
mapping from $\su(3)$ to \eqref{sl3R}.  It is natural to suppose that
the restriction of this mapping to each of $\ss$ and $\pp$ separately
is a \emph{multiple of the identity}. We therefore suppose that
\begin{equation}\label{natural}
\bbb\colon e_i\mapsto a E_i^a + s E_i^s,
\end{equation}
where $E_i$ is the matrix associated to $e_i$ (for example,
$E_1=E_{11}-E_{33}$), and (with some abuse of notation) the
coefficients $a,s$ are to be determined. More explicitly,
\begin{equation}\label{bbb}
\bbb=\hbox{\large$\left(\bbb_i^j\right)$} =
\left(\begin{array}{ccc} s(e^1-\frac1\sq e^2) & s e^3+ae^6 &
  se^4+ae^7\\[9pt] se^3-ae^6 & \frac2\sq se^2 & se^5+ae^8 \\[9pt]
se^4-ae^7 & se^5-ae^8 & -s(e^1+\frac1\sq e^2)
\end{array}\right).
\end{equation}
By construction, $\tr(\beta)=0$.

With this set-up, we can state

\begin{theorem}\label{main}
 The compact \aqH\ manifold $\{\SU(3),g_\lambda,\cI_\lambda\}$
 satisfies the ideal condition {\rm(2)} of Theorem~\ref{qK} if and
 only if $\lambda=\pm\sqrt{\frac3{20}}$. The resulting structure is
 invariant by the action of $\SU(3)$ on the left and $\SO(3)$ on the
 right, and its intrinsic torsion is $\SO(3)$-invariant.
\end{theorem}

\begin{proof} This is a direct computation. Solving the equations
\eqref{ideal} with the $\omega_i$ as in Proposition~\ref{ooo} and the
$\bbb_i^j$ as in \eqref{bbb}, we first find that a necessary condition
is \[\textstyle
 a= 1+\frac{16}3\lambda^2,\qquad s=-2\lambda.\]
Once these values are assigned, the remaining equations are satisfied
by taking $\lambda^2=3/20$.

By construction, the forms $e^i$ are left invariant, so all the
structures considered in this paper are invariant by left
translation. Right translation by $g\in \SU(3)$ can then be
identified with action of $\mathrm{Ad}(g)$ on the Lie algebra
$\su(3)$. In our case, we are
  free to take $g\in \SO(3)$, as defined in \eqref{SO3}.

We already know that the intrinsic torsion belongs to $\cW_1$.
The fact that it belongs to the 1-dimensionsinal subspace $S^0H$
follows because the intrinsic torsion is completely determined by
the map \eqref{natural} that is itself $\SO(3)$-equivariant.
\end{proof}

\smallbreak\noindent{\it Remarks. 1.} The two choices of sign for
$\lambda$ give a different quaternionic action and ideal structure
for the same metric $g_\lambda$.  Since $(\bbb^i_j)$ is not
anti-symmetric, the resulting structure on $\SU(3)$ is not
quaternionic K\"ahler. Note that $\SU(3)$ cannot in any case admit an
\aqH\ structure with $d\Omega=0$ since otherwise $[\Omega]$ would be a
non-zero element in cohomology, but $b_4(\SU(3))=0$.

{\it2.} The matrix $B$ of curvature 2-forms defined by \eqref{Bij}
with $\bbb$ in \eqref{bbb} will reflect the overall $\SO(3)$
invariance. One finds that, if the expression for $\omega_3$ in
Proposition~\ref{ooo} is written as $\tau_0+\lambda\tau_1
+\lambda^2\tau_2$, then
\[ B_2^1-B_1^2= -4(a+s^2)\tau_0 - 3a(a-1)\tau_2.\]
The symmetric coefficients are a bit more complicated, but the
diagonal ones are given by
\[ B^1_1 = c(e^{36} + e^{47}),\quad B^2_2=c(e^{58}-e^{36}),\quad
B^3_3 = -c(e^{47}+e^{58})\] where $c=2(a-1)s$. Assigning values to the
constants as in the proof of Theorem~\ref{main} does not eliminate any
terms.

\smallbreak

We conclude with some observations concerning \emph{integrable}
quaternionic structures. The Lie group $\SU(3)$ was shown by
Spindel, Servin, Troost \& Van Proeyen \cite{Spindel} to admit a
hypercomplex structure. In the treatment of Joyce \cite{Joyce}
this hypercomplex structure arises from a 3-dimensional subalgebra
$\su(2)$ inequivalent to $\ss$. This provides $\SU(3)$ with \aqH\
structures with intrinsic torsion in $\cW_3\oplus\cW_4$, but not
directly related to our construction. In our case, $\SU(3)$ cannot
admit an $\SO(3)$-invariant quaternionic structure. Such a
structure would necessarily have torsion in $\cW_3$ and satisfy
\begin{equation}\label{Omegao}
d\Omega\wedge\omega_i=0,\quad i=1,2,3,
\end{equation}
equations that can never be compatible with \eqref{glambda} if
$\lambda\in\R$.

Our computations can however be performed equally for the Lie
group $\SL(3,\R)$; it suffices to repeat everything with complex
coefficients. For this group, the situation is reversed; it turns
out that $\SL(3,\R)$ does not admit a structure of the type
described in Theorem~\ref{main}, but \eqref{Omegao} can be solved
when the analogue of the parameter $\lambda$ takes on the values
$\pm\frac12$. In this way, $\SL(3,\R)$ becomes a quaternionic
manifold, and admits a compatible Hermitian structure with torsion
in $\cW_3$.

\medbreak\small

\section*{Acknowledgements}

The author would like to thank Simon Salamon and Simon Chiossi for
enlightening discussions and useful suggestions, and the
Department of Mathematics at the Politecnico di Torino for its
hospitality during the preparation of this work. The latter was
supported by the Spanish Ministry of Science and Education (MEC)
and by the Spanish Foundation for Science and Technology (FECYT)
through a postdoctoral fellowship and research contract associated
with the project --2007-0857.

\parindent0pt

\end{document}